\documentclass[11pt]{amsart}

\usepackage{amsmath,amsthm,epsfig,graphicx,subfigure}

\usepackage{amssymb}

\newcommand{\C}{{\mathbb C}}       
\newcommand{\R}{{\mathbb R}}       
\newcommand{\N}{{\mathbb N}}       %
\newcommand{\Z}{{\mathbb Z}}       
\newcommand{\DD}{{\mathcal D}}
\newcommand{\HH}{{\mathcal H}}

\newcommand{\stm}{{\setminus}}

\newcommand{\CC}{{\mathcal C}}

\newcommand{\diam}{{\rm diam}}
\newcommand{\dist}{{\rm dist}}
\newcommand{\ds}{\displaystyle }
\newcommand{\ha}{\mathcal{H} }

\newcommand{\ra}{\rightarrow}
\renewcommand{\a}{\alpha}

\newcommand{\rf}[1]{{(\ref{#1})}}
\newcommand{\riesz}{{R^s_{\ve_1,\ve_2}}}
\newcommand{\rieszz}{{R^1_{\ve_1,\ve_2}}}
\newcommand{\riesze}{{R^s_{\ve_1}}}
\newcommand{\supp}{\operatorname{supp}}

\newcommand{\vphi}{{\varphi}}
\newcommand{\ve}{{\varepsilon}}
\newcommand{\vv}{{\vspace{2mm}}}
\newcommand{\vvv}{{\vspace{3mm}}}
\newcommand{\wt}[1]{{\widetilde{#1}}}

\newcommand{\wolf}{{\dot{W}^{\mu}_{\frac{2}{3}(d-s),\frac32}}}
\newcommand{\wolfr}{{\dot{W}^{\chi_{B(0,r)}\mu}_{\frac{2}{3}(d-s),\frac32}}}
\newcommand{\wolfrk}{{\dot{W}^{\chi_{B(0,\wt r_k)}\mu}_{\frac{2}{3}(d-s),\frac32}}}
\newcommand{\meas}{{\measuredangle}}

\newtheorem{theorem}{Theorem}[section]
\newtheorem*{maintheorem*}{Theorem 1.1}
\newtheorem*{theorema*}{Theorem A}
\newtheorem*{theoremb*}{Theorem B}
\newtheorem*{theoremc*}{Theorem C}
\newtheorem*{theoremd*}{Theorem D}
\newtheorem{lemma}[theorem]{Lemma}

\newtheorem{coro}[theorem]{Corollary}
\newtheorem{propo}[theorem]{Proposition}

\theoremstyle{definition}

\theoremstyle{remark}
\newtheorem{remark}[theorem]{Remark}

\numberwithin{equation}{section}

\pdfoutput=1 

\begin{document}

\title{Square functions of fractional homogeneity and Wolff potentials}

\author{Vasilis Chousionis, Laura Prat and Xavier Tolsa}

\address{Vasilis Chousionis. Department of Mathematics \\ University of Connecticut \\ Storrs \\ CT 06269, USA }

\address{Department of Mathematics and Statistics \\ University of Helsinki \\ P. O. Box 68 \\ FI-00014, Finland}
\email{vasileios.chousionis@helsinki.fi}

\address{Laura Prat. Departament de Ma\-te\-m\`a\-ti\-ques, Universitat Aut\`onoma de Bar\-ce\-lo\-na, Catalonia}
\email{laurapb@mat.uab.cat}

\address{Xavier Tolsa. Instituci\'{o} Catalana de Recerca i Estudis Avan\c{c}ats (ICREA) and Departament de Ma\-te\-m\`a\-ti\-ques, Universitat Aut\`onoma de Bar\-ce\-lo\-na, Catalonia}
\email{xtolsa@mat.uab.cat}

\thanks{V.C.\ was funded by the Academy of Finland Grant SA 267047. Also, partially supported by the ERC Advanced Grant 320501, while visiting Universitat Aut\`onoma de Bar\-ce\-lo\-na. 
L.P. and X.T. were partially supported by the grants MTM-2013-44304-P (MICINN, Spain) and 
 2014-SGR-75 (Catalonia).
X.T.\ was also supported by 
the ERC grant 320501 of the European Research
Council (FP7/2007-2013) and by Marie Curie ITN MAnET (FP7-607647).}

\maketitle

\begin{abstract}
 In this paper it is shown that for any measure $\mu$ in $\R^d$ and for a non-integer $0<s<d$, the Wolff energy 
 $\ds\iint_0^\infty\left(\frac{\mu(B(x,r))}{r^s}\right)^2\,\frac{dr}{r}d\mu(x)$ 
 is comparable to 
 $$\iint_0^\infty\left(\frac{\mu(B(x,r))}{r^s} - \frac{\mu(B(x,2r))}{(2r)^s}\right)^2\,\frac{dr}rd\mu(x),$$
 unlike in the case when $s$ is an integer. We also study the relation with the $L^2-$norm of 
 $s$-Riesz transforms, $0<s<1$, and we provide a counterexample in the integer case. 

\end{abstract}

\section{Introduction}

Let $\mu$ be a Radon measure in $\R^d$. Given $x\in\R^d,\;r>0$ and $s>0$, we set
\begin{equation}\label{defdifdens}
\Delta^s_\mu(x,r) := \frac{\mu(B(x,r))}{r^s} - \frac{\mu(B(x,2r))}{(2r)^s}.
\end{equation}
The main result of this paper shows the comparability between the squared $L^2(\mu)-$norm of a square function involving the difference of densities \eqref{defdifdens} 
and the Wolff energy of the measure $\mu$, for measures $\mu$ in $\R^d$ and non-integer $s$, $0<s<d$. 
Before stating precisely the theorem, we need to introduce some notation. 

Let $\theta^s_\mu(B(x,r))$ be the average $s$-dimensional density of $\mu$ on $B(x,r)$, that is 
$$\theta^s_\mu(B(x,r))=\frac{\mu(B(x,r))}{r^s},$$
so that 
$$\Delta^s_\mu(x,r) =
\theta^s_\mu(B(x,r)) - \theta^s_\mu(B(x,2r)).$$

Let $\a>0$ and $p \in (0,\infty)$ such that $\a p\in (0,d)$. The Riesz capacity $\dot{C}_{\alpha,p}$ of $E \subset \R^d$ is defined as 
$$\dot{C}_{\alpha,p}(E)=\sup_{\mu \in M(E)} \left(\frac{\mu(E)}{\|I_\a \ast \mu\|_{p'}}\right)^p, \quad \quad I_\a (x)= \frac{A_{d,\a}}{|x|^{d-\a}},$$
where $M(E)$ is the set of positive Radon measures  supported on $E$ and as usual $p'=p/(p-1)$. In nonlinear potential theory Riesz capacities occur naturally in the study of Sobolev spaces, for example they measure exceptional sets for functions in these function spaces, see e.g \cite{ah}. 

For $\a$ and $p$ as before the Wolff potential of a positive Radon measure $\mu$ is defined as 
$$\dot{W}^{\mu}_{\a,p}(x)=\int_0^\infty\left(\frac{\mu(B(x,r))}{r^{d-\a \, p}}\right)^2\,\frac{dr}{r},\quad\;x\in\R^d,$$
and its Wolff energy is
$$\int\dot{W}^{\mu}_{\a,p}(x)d\mu(x).$$
Riesz capacities can be characterized via Wolff potentials, as a well known theorem of Wolff, see e.g. \cite[Theorem 4.5.4]{ah}, asserts that
$$C^{-1}\,\|I_\a \ast \mu\|^{p'}_{p'}\leq \int\dot{W}^{\mu}_{\a,p}(x)d\mu(x) \leq C \,\|I_\a \ast \mu\|^{p'}_{p'}$$ 
where $C$ is a constant depending only on $d, \a$ and $p$.

In this paper we consider Wolff potentials with indices $\frac{2}{3}(d-s)$, $\frac32$, where $0<s<d$. Notice that  Wolff potentials with these choice of indices are related to the $s$-dimensional density of $\mu$, in particular
$$\dot{W}^{\mu}_{\frac{2}{3}(d-s),\frac32}(x)=\int_0^\infty\left(\frac{\mu(B(x,r))}{r^s}\right)^2\,\frac{dr}{r}=\int_0^\infty \theta_\mu^s(B(x,r))^2\,\frac{dr}{r},\quad\;x\in\R^d.$$
We further remark that these potentials are related to the Calder\'on-Zygmund capacities associated with the vector valued Riesz kernels $K^s(x)=x/|x|^{1+s}, x \in \R^d$, see e.g. the excellent survey \cite{ev} or \cite{mpv}.

Our main result reads as follows:
\begin{theorem} \label{teomain1}
Let $\mu$ be a Radon measure on $\R^d$ and $0<s<d$ be non-integer. Then
\begin{equation}
\label{squarefunction}
\iint_0^\infty\Delta^s_\mu (x,r)^2\,\frac{dr}rd\mu(x) \approx \int \dot{W}^{\mu}_{\frac{2}{3}(d-s),\frac32} (x)\,d \mu(x).
\end{equation}
\end{theorem}
The notation $A\approx B$ means that there is an absolute constant $c>0$, depending on $d$ and $s$ (and sometimes on other fixed parameters), such that $c^{-1}A\le B\le cA$.

We remark that for integer $0<s<d$ the estimate \eqref{squarefunction} does not hold; just let $\mu=\ha^s |_V$, the restriction of the $s$-dimensional Hausdorff measure to any affine $s$-plane $V$. This is connected to the well known 
theorem of Marstrand \cite{Marstrand},  which asserts that for $s>0$, given a Radon measure $\mu$ on $\R^d$ such that the density $\lim_{r\to 0}\theta^s_\mu(B(x,r))$  exists and is positive and finite 
 in a set of positive $\mu$ measure, $s$ must be an integer.
  
In the context of integer $s$, there are also results relating rectifiability and the kind of square functions appearing in the left hand side of \eqref{squarefunction}. 
In \cite{tatianatolsa} it is shown that, for Radon measures $\mu$ in $\R^d$ with $\mu$-almost everywhere positive and finite lower and upper $s$-dimensional 
densities ($s\in\N$ here) the fact that $\mu$ is $s$-rectifiable is equivalent to the $\mu-$almost everywhere finiteness of 
$\int_0^1\Delta^s_\mu (x,r)^2\,\frac{dr}r$ and also to the fact that $\lim_{r\to 0}\Delta^s_\mu(x,r)=0$ $\mu$-almost everywhere. It is worth also saying that the 
first just mentioned equivalence from \cite{tatianatolsa} is a pointwise version of a previous result in \cite{CGLT}, which characterizes the so called uniform 
rectifiability. In fact, in \cite[Lemma 3.1]{CGLT} a blow up argument is used, which turns to be one of the main ingredients in the proof of Theorem 1.1 (see Lemma \ref{blowup}).

Let us remark that a suitable $p$-th version of Theorem \ref{teomain1}  
 holds for $p\in[1,\infty)$. Indeed, almost the same proof yields that,
for $0<s<d$ non-integer and such $p$,
$$\iint_0^\infty|\Delta^s_\mu (x,r)|^p\,\frac{dr}rd\mu(x) \approx 
\iint_0^\infty \left(\frac{\mu(B(x,r))}{r^s}\right)^p
\,\frac{dr}rd\mu(x),$$
with the comparability constant depending only on $s$, $d$ and $p$. Notice that
$\int_0^\infty \left(\frac{\mu(B(x,r))}{r^s}\right)^p
\,\frac{dr}r$ coincides with the Wolff potential $\dot{W}^{\mu}_{\frac{p}{p+1}(d-s),p} (x)$, so that
$$\iint_0^\infty|\Delta^s_\mu (x,r)|^p\,\frac{dr}rd\mu(x)\approx 
\int \dot{W}^{\mu}_{\frac{p}{p+1}(d-s),p} (x)\,d \mu(x).$$
Nevertheless, we think that the case $p=2$ is by far the most important one because of the connection with rectifiability
mentioned above and because of the relationship with Riesz transforms.

In fact, Theorem \ref{teomain1} answers a question of F. Nazarov (private communication), motivated by an open problem concerning the comparability between the Wolff energy of a measure $\mu$ in $\R^d$ 
and the squared $L^2(\mu)-$norm of the $s$-Riesz transform with respect to $\mu$, for non-integer $0<s<d$. To state the problem in detail, we
need to introduce some additional notation and background. For $0<s<d$, consider the signed vector valued Riesz kernels $$K^s(x)=\frac{x}{|x|^{1+s}},\;x\in\R^d,\;x\neq 0.$$
The $s$-Riesz transform of a real Radon measure $\mu$ with compact support is
$$R^s\mu(x)=\int K^s(y-x)d\mu(y)$$
whenever the integral makes sense. To avoid delicate problems with convergence, one considers the 
truncated $s-$Riesz transform of $\mu$, which is defined as $$R_{\varepsilon}^s\mu(x)=\int_{|y-x|>\varepsilon}K^s(y-x)d\mu(y),\;\;x\in\R^d,\;\varepsilon>0.$$
One says that $R^s\mu$ is bounded in $L^2(\mu)$ if the truncated Riesz transforms $R^s_{\varepsilon}\mu$ are bounded in $L^2(\mu)$ uniformly in $\varepsilon$.

It was shown in \cite{mpv} that given a finite Radon measure $\mu$ in $\R^d$ with growth $s$, $0<s<1$, that is, $\mu$ satisfying $\mu(B(x,r))\leq c_\mu r^{s}$ 
for all $x\in\R^d$, $r>0$ and some constant $c_\mu>0$,
one has  
\begin{equation}\label{conjecture}
\int\dot{W}^{\mu}_{\frac{2}{3}(d-s),\frac32}(x)d\mu(x)\approx\sup_{\varepsilon>0}\int|R_\varepsilon^s\mu(x)|^2d\mu(x),\qquad0<s<1.
\end{equation}
It is known that for the positive integers $s$ this comparability is false, while 
for non-integer $s\in(1,d)$ it is an open problem to prove (or disprove) it. There are some (very) partial results in this direction. In \cite{ENV} it is shown that for  $s  \in (0,d)$, the Wolff energy controls the $L^2$-norm of the $s$-Riesz transform; in \cite{JNV} it is proved that for $s\in (d-1,d)$, $d\geq 2$, boundedness of the $s$-Riesz 
transform of $\mu$ implies $\mu$-almost everywhere finiteness of a non-linear potential of exponential type. In the special case of measures supported on Cantor type sets, 
the comparability \eqref{conjecture} has been proven for all $0<s<d$ (see \cite{ev1}, \cite{Tolsawolff} and \cite{RT}).
Since the square function on the left hand side of \eqref{squarefunction} has a cancellative nature while the Wolff potential does not, one could think of Theorem 1.1 as being, in a sense, an intermediate stage towards the proof of \eqref{conjecture} for non-integer $1<s<d$,  since the $s$-Riesz transform also has an analogous cancellative nature.

The plan of the paper is the following. In Section \ref{sec:bddur} we prove Theorem 1.1. Section \ref{secriesz} is devoted to the study of the 
relation between the $L^2$-norm of the $s$-Riesz transform, the Wolff energy and the square function on the left hand side of \eqref{squarefunction}, for $0<s<1$. In the final section we construct a measure with linear growth and infinite Wolff energy 
 for which the $L^2(\mu)$-norm of the $1$-Riesz transform with respect to $\mu$ is finite and much bigger than $\displaystyle{\iint_0^\infty\Delta_\mu^1(x,r)^2\frac{dr}{r}d\mu(x)}$.

Throughout the paper, the letters $c,\;C$ will stand for absolute constants (which may depend on $d$ and $s$) that may change at different occurrences.
\\

\paragraph{\bf Acknowledgments.} We extend our sincere thanks to the referees for several useful comments 
which improved the readability of our paper.




\section{Densities and Wolff potentials} \label{sec:bddur}

The aim of this section is to prove our main result, Theorem $1.1$.
Its proof follows easily once we have at our disposal the following proposition.

\begin{propo}\label{lemcpt1}
Let $s$ be positive and non-integer. Then there exists some $\delta \in (0,1)$ such that for every Radon measure $\mu$ on $\R^d$ and every open ball $B_0\subset \R^d$ of radius $r_0$, 
$$\int_{\delta r_0}^{\delta^{-1}r_0} \int_{\delta^{-1}B_0}\Delta^s_\mu (x,r)^2\,d\mu(x)\frac{dr}r
\geq c(\delta)\, \theta^s_\mu({B_0})^2 \,\mu({B_0}),$$
for some constant $c(\delta)$.
\end{propo}

\begin{proof}[Proof of Theorem 1.1]It is enough to prove that
\begin{equation}
\iint_0^\infty\Delta^s_\mu (x,r)^2\,\frac{dr}rd\mu(x) \gtrsim \int \dot{W}^{\mu}_{\frac{2}{3}(d-s),\frac32} (x)\,d \mu(x),
\end{equation}
since the remaining inequality is immediate. 

Let $\DD$ denote the usual lattice of dyadic cubes of $\R^d$, and let $\DD_k\subset \DD, k \in \Z,$ be the subfamily of the dyadic cubes with 
side length $\ell(Q)=2^{k}$. For $Q\in\DD$, let $B_Q=B(x_Q, r(Q))$ where $x_Q$ is the center of $Q$ and $r(Q)=(2+\sqrt{d})\ell(Q)$. 
Using Fubini's theorem and Proposition \ref{lemcpt1} one easily sees that
\begin{equation}
\begin{split}
\label{wolfpr}
\int \wolf (x)\, d \mu (x) &= \sum_{k \in \Z} \sum_{Q \in \DD_k} \int_Q \int_{2^k}^{2^{k+1}} \left(\frac{\mu(B(x,r))}{r^s}\right)^2\,\frac{dr}{r} d \mu (x)\\
&\lesssim  \sum_{Q \in \DD} \theta^s_\mu(B_Q)^{2} \mu(B_Q)\\
&\lesssim \sum_{Q \in \DD}\int_{\delta \, r(Q)}^{\delta^{-1} \, r(Q)} \int_{\delta^{-1}B_Q}\Delta^s_\mu (x,r)^2\,d\mu(x)\frac{dr}r,
\end{split}
\end{equation}
for some $\delta\in (0,1)$.
Given $k \in \Z$ the family of balls $\{\delta^{-1} B_Q\}_{Q \in \DD_k}$ has finite overlap (which depends only on $\delta$ and on the ambient dimension $d$). Therefore using \eqref{wolfpr} we get,
\begin{equation*}
\begin{split}
\int \wolf (x)\, d \mu (x)&\lesssim \sum_{k \in \Z} \sum_{Q \in \DD_k}\int_{\delta^{-1}B_Q} \int_{\delta \, (2+\sqrt{d}) \, 2^k}^{\delta^{-1} \, (2+\sqrt{d}) \, 2^k} \Delta^s_\mu (x,r)^2\frac{dr}r\,d\mu(x) \\
&\lesssim \int \sum_{k \in \Z} \int_{\delta \, (2+\sqrt{d}) \, 2^k}^{\delta^{-1} \, (2+\sqrt{d}) \, 2^k} \Delta^s_\mu (x,r)^2\frac{dr}r\,d\mu(x) \\
&\lesssim \iint_0^\infty \Delta^s_\mu (x,r)^2\frac{dr}r\,d\mu(x).
\end{split}
\end{equation*}
\end{proof}
Before providing the proof of Proposition \ref{lemcpt1} we need some auxiliary results and additional notation. For any Borel function $\vphi: [0, \infty) \ra \R$ let 
$$\vphi_t(x)=\frac{1}{t^s} \vphi \left(\frac xt \right), \, t>0,$$
and define
$$\Delta^s_{\mu,\vphi} (x,t):= \int\bigr(\vphi_t (|y-x|)-\vphi_{2t}(|y-x|)\bigr)\,d\mu(y),$$
whenever the integral makes sense. 

\begin{lemma}\label{lemconvex}
 Let $\vphi:[0,\infty)\to\R$ be a $\CC^\infty$ function supported in 
$[0,2]$ which is constant in $[0,1/2]$. Let $x\in\R^d$ and $0\leq r_1<r_2$. Then
$$\int_{r_1}^{r_2} |\Delta^s_{\mu,\vphi}(x,r)|\,\frac{dr}r \leq c
\int_{r_1/2}^{2r_2} |\Delta^s_{\mu}(x,r)|\,\frac{dr}r,$$
where $c$ depends only on $\vphi$.
\end{lemma}

\begin{proof}
This follows by writing $\vphi$ as a suitable convex combination of functions of the form $\chi_{[0,r]}$.
For completeness we show the details. For $t\geq0$ and $R>0$, we write
$$\frac1{R^s}\vphi\left(\frac tR\right) = -\int_0^\infty \frac1{R^{s+1}}\,\vphi'\left(\frac rR\right) \,\chi_{[0,r]}(t)\,dr,$$
so that, by Fubini and changing variables,
\begin{equation}\label{eqff22}
\begin{split}
\Delta^s_{\mu,\vphi}(x,R) & = - \int_0^\infty \!\frac1{R^{s+1}} \,\vphi'\!\left(\frac rR\right) \chi_{[0,r]}(|\cdot|)* \mu(x) \,dr \\
&\quad \quad \quad
+ \int_0^\infty \!\!\frac1{(2R)^{s+1}}\, \vphi'\!\left(\frac r{2R}\right) \chi_{[0,r]}(|\cdot|)* \mu(x) \,dr\\
& = 
- \int_0^\infty  \vphi'(t) \left(\frac1{R^s}\,\chi_{[0,tR]}(|\cdot|)* \mu(x) - \frac1{{(2R)}^s}\,\chi_{[0,2tR]}(|\cdot|)* \mu(x)
\right) \,dt \\
&=
- \int_{1/2}^2  t^s\,\vphi'(t) \,\Delta^s_\mu(x,tR) \,dt,
\end{split}
\end{equation}
taking into account that $\vphi'$ is supported on $[1/2,2]$ in the last identity.
As a consequence we get
$$|\Delta^s_{\mu,\vphi}(x,r)|\leq \left|
\int_{1/2}^2  t^s\,\vphi'(t) \,\Delta^s_\mu(x,tr) \,dt\right|
\lesssim \int_{1/2}^2  |\Delta^s_\mu(x,tr)| \,dt = \int_{r/2}^{2r}  |\Delta^s_\mu(x,u)| \,\frac{du}r.$$
Thus
\begin{align*}
\int_{r_1}^{r_2} |\Delta^s_{\mu,\vphi}(x,r)|\,\frac{dr}r & \lesssim  
\int_{r_1}^{r_2} 
\int_{r/2}^{2r} |\Delta^s_\mu(x,u)| \,du\,\frac{dr}{r^2} \lesssim
\int_{r_1/2}^{2r_2} 
 |\Delta^s_\mu(x,u)| \,\frac{du}u.
\end{align*}
\end{proof}

\begin{remark}\label{realf}
Notice that if $\vphi:[0, \infty) \ra \R$ is a smooth function vanishing at infinity, then as in \eqref{eqff22} we get 
$$\Delta^s_{\mu,\vphi}(x,R)=-\int_{0}^\infty  t^s\,\vphi'(t) \,\Delta^s_\mu(x,tR) \,dt.$$
\end{remark}

\begin{lemma}\label{lemfac31} 
Let $s$ be positive and non-integer and let $\mu$ be a non-zero Radon measure in $\R^d$. 
Then $\Delta^s_\mu(x_0,r_0) \neq 0$ for some $x_0 \in \supp (\mu)$ and $r_0>0$. 
\end{lemma}

\begin{proof} By way of contradiction suppose that $\Delta^s_\mu(x,r)=0$ for all $x \in \supp (\mu)$ and all $r>0$. We will first show that in that case the measure $\mu$ 
is $s$-AD regular and we will then proceed as in the proof of \cite[Lemma 3.9]{CGLT}. Recall that $\mu$ is
called $s$-Ahlfors-David regular, or $s$-AD regular, if for some constant $c_\mu>0$, 
$$c_\mu^{-1}r^{s}\leq \mu(B(x,r)) \leq c_\mu\,r^{s}\quad\mbox{for all $x\in\supp(\mu)$,
$0<r\leq\diam(\supp(\mu))$}.$$ 

To prove the $s$-AD-regularity of $\mu$, assume for simplicity that $0\in\supp\mu$. Since $\Delta^s_\mu(0,r)=0$ for all $r>0$, we deduce that
$\mu(B(0,2^n))= 2^{ns}\,\mu(B(0,1))$ for all $n\geq1$.
For $x\in\supp(\mu)\cap B(0,2^{n-1})$ and any integer $m\leq n$, using now that $\Delta^s_\mu(x,r)=0$ for all $r>0$, we infer that $\mu(B(x,2^m))= 2^{(m-n)s}\mu(B(x,2^n))$.
Since $B(0,2^{n-1})\subset B(x,2^n)\subset B(0,2^{n+1})$, we have
$$2^{(n-1)s} \mu(B(0,1))\leq \mu(B(x,2^n)) \leq 2^{(n+1)s} \mu(B(0,1)).$$
Thus
$$c_0\,2^{(m-1)s}\leq \mu(B(x,2^m)) \leq c_0\,2^{(m+1)s} ,$$
with $c_0=\mu(B(0,1))$.
Since $n$ can be taken arbitrarily large and the preceding estimate holds for all $m\leq n$,
the $s-$AD regularity of $\mu$ follows.

Let $\vphi(u)=e^{-u^2}, \, u \geq 0$. Then by Remark \ref{realf} it follows that $\Delta^s_{\mu,\vphi}(x,r)=0$ for all $x \in \supp(\mu)$ and for all $r>0$. This is equivalent to
$$\phi_r \ast \mu (x)-\phi_{2r} \ast \mu (x)=0$$ 
for all $x \in \supp(\mu)$ and for all $r>0$, where $\phi:\R^d \ra \R$ is defined by $\phi(y)=e^{-{|y|^2}}$. In particular
\begin{equation}\label{eqkw433}
\phi_{2^{-k}} *\mu(x) - \phi_{2^k}*\mu(x) =0\qquad \mbox{ for all $k>0$ and all $x\in\supp(\mu)$.}
\end{equation}
Now consider the function $F:\R^d\to \R$ given by
$$F(x) = \sum_{k>0}2^{-k}\Bigl(\phi_{2^{-k}} *\mu(x) - \phi_{2^k}*\mu(x)\Bigr)^2.$$
Taking into account that $|\phi_{2^{-k}} *\mu(x) - \phi_{2^k}*\mu(x)|\leq c$ for all $x\in\R^d$ and $k\in\N$, it is clear 
that $F(x)<\infty$ for all $x\in\R^d$, and so $F$ is well defined. 
Moreover, by \rf{eqkw433} we have $F =0$ on  $\supp(\mu)$. 

Now we claim that $F(x)>0$ for all $x\in\R^d\setminus \supp(\mu)$.
Indeed, it follows easily that
$$\lim_{k\to\infty}\phi_{2^{-k}}*\mu(x)= 0\qquad \mbox{for all $x\in\R^d\setminus \supp(\mu)$,}$$
while, by the $s$-AD-regularity of $\mu$,
$$\liminf_{k\to\infty}\phi_{2^{k}}*\mu(x)\geq c\, c_0\qquad \mbox{for all $x\in\R^d$.}$$
Thus if $x\in\R^d\setminus \supp(\mu)$ we have $\phi_{2^{-k}} *\mu(x) - \phi_{2^k}*\mu(x)\neq0$ for all large enough $k>0$,
which implies that $F(x)>0$ and proves our claim. We have thus shown that  $\supp(\mu) = F^{-1}(0)$.


Next we will prove that the zero set of $F$ is a real analytic variety. It is enough to check that
$\phi_{2^{-k}} *\mu - \phi_{2^k}*\mu$ is a real analytic function for each $k>0$, because the
zero set of a real analytic function is a real analytic variety and the intersection
of any family  of real analytic varieties is again a real analytic variety; see \cite{nar}.   So it is
enough to show that $\phi_r *\mu$ is a real analytic function for every $r>0$. To this end, we consider 
the function $f:\C^d\to\C$ defined by
$$f(z_1,\ldots,z_d)= \frac1{r^n}\int \exp\Biggl(-r^{-2}\sum_{i=1}^d (y_i-z_i)^2\Biggr)\,d\mu(y).$$
It is easy to check that $f$ is well defined and holomorphic in the whole $\C^d$, and thus
$\phi_r *\mu = f|_{\R^d}$ is real analytic.

Therefore we have shown that $\supp (\mu)$ is an analytic variety, in particular this implies that $\supp(\mu)$ has Hausdorff dimension $n$ for some $n \in \N$.  Since $\mu$ is $s$-AD regular,  $\supp(\mu)$ has non-integer Hausdorff dimension and we have thus reached a contradiction. 
\end{proof}

The following blow-up lemma is essential for the proof of Proposition \ref{lemcpt1}. The proof is inspired by the proof of \cite[Lemma 3.1]{CGLT}
\begin{lemma}\label{blowup}
Let $s$ be a positive and non-integer real number. There exists some $\delta>0$ such that for every Radon measure $\mu$ in $\R^d$ which satisfies $1\leq \mu(\bar B(0,1))\leq\mu(B(0,2))\leq 2^{5s+2},$ the following estimate holds
$$\int_{\delta}^{\delta^{-1}}\!\!
 \int_{x\in B(0,\delta^{-1})} |\Delta^s_{\mu} (x,r)|\,d\mu(x)\,\frac{dr}r \geq \delta^{1/2}.$$
\end{lemma}

\begin{proof}
By way of contradiction suppose that for each $m\geq 1$ there exists a Radon
 measure $\mu_m$ such that $1\leq \mu_m(\bar B(0,1))\leq\mu_m(B(0,2))\leq 2^{5s+2}$ which satisfies
\begin{equation}\label{eqass32}
\int_{1/m}^{m}
 \int_{x\in  B(0,m)} |\Delta^s_{\mu_m} (x,r)|\,d\mu_m(x)\,\frac{dr}r \leq \frac1{m^{1/2}}.
\end{equation}

We will first show that the sequence $\{\mu_m\}$ has a subsequence $\{\mu_{m_j}\}$ which converges weakly * (i.e. when tested against 
compactly supported continuous functions) to a measure $\mu$.
This follows from \cite[Theorem 1.23]{Mattila-llibre} once we show that $\mu_m$ is uniformly bounded on compact sets. That is,
for any compact $K\subset\R^d$, $\sup_m\mu_m(K)<\infty$. To prove this, 
for $n\geq 4$, $1/4<r<1/2$, and $x\in B(0,1)$, 
we write
\begin{align*}
\frac{\mu_m(B(0,2^{n-3}))}{2^{(n+2)s}}  \leq
\frac{\mu_m(B(x,2^nr))}{(2^nr)^s}& \leq \sum_{k=1}^n |\Delta^s_{\mu_m}(x,2^{k-1}r)| + \frac{\mu_m(B(x,r))}{r^s}\\
& \leq \sum_{k=1}^n |\Delta^s_{\mu_m}(x,2^{k-1}r)| + 4^s\,\mu_m(B(0,2)).
\end{align*}
Integrating this estimate with respect to $\mu_m$ on $B(0,1)$ and with respect to $r\in[1/4,1/2]$, 
using \rf{eqass32} for $m$ big enough
we obtain
\begin{align*}
\mu_m(B(0,2^{n-3})) &\leq 2^{(n+2)s} \left[\sum_{k=1}^n \frac1{\log2}\int_{1/4}^{1/2}\!\!\int_{B(0,1)}
\!\!|\Delta^s_{\mu_m}(x,2^{k-1}r)|
d\mu_m(x) \frac{dr}r + 4^s\mu_m(B(0,2))\right]\\
&\leq c(n),
\end{align*}
which proves the uniform boundedness of $\mu_m$ on compact sets.

Our next objective consists in proving that that $\Delta^s_\mu(x,r)=0$ for all $x\in\supp(\mu)$ and all $r>0$. Once this is done, the lemma would follow from
Lemma \ref{lemfac31} since it is easy to check that $\mu(\bar B(0,1))\geq1$, and thus $\mu$ is not identically zero.

To prove that $\Delta^s_\mu(x,r)$ vanishes identically on $\supp\mu$ for all $r>0$, we will show first that, given any $\CC^\infty$ function $\vphi:[0,\infty)\to\R$ which is supported in 
$[0,2]$ and constant in $[0,1/2]$, we have
\begin{equation}\label{eqlimit}
\int_{0}^\infty\!
 \int_{x\in\R^d} |\Delta^s_{\mu,\vphi} (x,r)|\,d\mu(x)\,\frac{dr}r=0.
 \end{equation}
The proof of this fact is elementary.  
Fix  $m_0$ and 
 let $\eta > 0$. 
 Set  $K = [2/m_0,\,m_0/2] \times \bar B(0,m_0)$.  
Now $\{y \to  \vphi_t(|x-y|)- \vphi_{2t}(|x-y|),~ (t,x) \in K\}$ is an equicontinuous family of continuous functions supported inside a fixed compact set. Hence setting, $\phi(x)=\vphi(|x|), \, x \in \R^d$,  we get that $(\phi_t- \phi_{2t}) * \mu_{m_j}(x)$ converges to $ (\phi_t- \phi_{2t}) * \mu(x)$ uniformly on $K$.  It therefore follows that 
\begin{equation*} \begin{split}
\iint_K  |\Delta^s_{\mu,\vphi} (x,t)|d\mu(x)\frac{dt}t&=\iint_K  |(\phi_t- \phi_{2t})  *\mu(x)|d\mu(x)\frac{dt}t \\
& = \lim_j \int_{2/m_0}^{m_0/2}
 \int_{x\in \bar B(0,m_0)} |(\phi_t- \phi_{2t}) *\mu_{m_j}(x)|d\mu_{m_j}(x)\frac{dt}t \\
 &= 
 \lim_j \int_{x\in \bar B(0,m_0)}  \int_{2/m_0}^{m_0/2}|\Delta^s_{\mu_{m_j},\vphi} (x,t)|\frac{dt}t\,d\mu_{m_j}(x) \\
 &\lesssim \lim_j 
 \int_{x\in \bar B(0,m_0)}\int_{1/m_0}^{m_0} |\Delta^s_{\mu_{m_j}} (x,t)|d\mu_{m_j}(x)\frac{dt}t=0
 \end{split}
\end{equation*}
by Lemma \ref{lemconvex} and \rf{eqass32}.
Since  this holds for any $m_0\geq 1$, our claim \eqref{eqlimit}  is proved.

Denote by $G$ the subset of those points $x\in\supp(\mu)$ such that
$$\int_{0}^\infty\!
 | \Delta^s_{\mu,\vphi} (x,r)|\,\frac{dr}r=0.$$
It is clear now that $G$ has full $\mu$-measure. By continuity, it follows that $\Delta^s_{\mu,\vphi} (x,r)=0$ for all $x\in
\supp\mu$ and all $r>0$. Finally, by taking a suitable sequence of $\CC^\infty$ functions $\vphi_k$ which
converge to $\chi_{[0,1]}$ we infer that $\Delta^s_{\mu} (x,r)=0$ for all $x\in\supp\mu$ and $r>0$. By Lemma \ref{lemfac31}, this is impossible.
\end{proof}

By renormalizing the preceding lemma we get:

\begin{lemma}\label{lemcompactbol}
Let $s$ be a positive and non-integer real number. There exists some $\delta>0$ such that for every Radon measure $\mu$ in $\R^d$ and every open
ball of radius $r_0$ such that
 $0<\mu(\bar B_0)\leq\mu(2B_0)\leq 2^{5s+2}\,\mu(\bar B_0),$ the following estimate holds
$$\int_{\delta\,r_0}^{\delta^{-1}\,r_0}\!\!
 \int_{x\in \delta^{-1} B_0} |\Delta^s_{\mu} (x,r)|\,d\mu(x)\,\frac{dr}r \geq \delta^{1/2}\,\frac{\mu(\bar B_0)^2}{r_0^s}.$$
\end{lemma}

\begin{proof}
Let $T:\R^d\to\R^d$ be an affine transformation which maps $\bar B_0$ to $\bar B(0,1)$. Consider the measure
$\sigma=\frac1{\mu(\bar B_0)}\,T_{\#}\mu$, where as usual $T_{\#}\mu(E):= \mu (T^{-1}(E))$, and apply the preceding lemma to $\sigma$.
\end{proof}

We can now complete the proof of Proposition \ref{lemcpt1}.

\begin{proof}[Proof of Proposition \ref{lemcpt1}] Let $B_0$ be an open ball of radius $r_0$ such that $\mu (B_0)>0$. 
Let $\delta\in (0,1)$ to be fixed below and let $k=k(\delta)$ be such that $2^{-k}\leq\delta< 2^{-k+1}$. If
\begin{equation*}
\int_{2^{-k-2}r_0}^{2^{k+2}r_0}\int_{2^{k+2}B_0} \Delta^s_{\mu} (x,r)^2 d \mu(x) \frac{dr}r > \delta^4 \,\theta^s_\mu({B_0})^2\, \mu({B_0})
\end{equation*}
we are done. Otherwise, there exists some $x \in B_0$ such that
\begin{equation}
\label{che1}
\int_{2^{-k-2}r_0}^{2^{k+2}r_0}\Delta^s_{\mu} (x,r)^2 \frac{dr}r \leq 2 \delta^4 \,\theta^s_\mu({B_0})^2.
\end{equation}
Notice also that after changing variables, for any $ n\in \Z$, we have
$$\int_{{r_0} /2}^{r_0}\Delta^s_{\mu} (x,2^n r)^2 \frac{dr}r =\int_{2^{n-1}r_0 }^{2^n r_0}\Delta^s_{\mu} (x,r)^2 \frac{dr}r.$$
Therefore 
\begin{equation}
\label{chv}
\int_{2^{-k-2}r_0}^{2^{k+2}r_0}\Delta^s_{\mu} (x,r)^2 \frac{dr}r=\int_{{r_0} /2}^{r_0} \sum_{n=-k-1}^{k+2}\Delta^s_{\mu}  (x,2^n r)^2 \frac{dr}r.
\end{equation}
Using \eqref{che1}, \eqref{chv} and applying Chebyshev's inequality  with respect to the measure $dt/t$ we find some $t \in [r_0/2,r_0]$ such that
\begin{equation*}
\sum_{n=-k-1}^{k+2}\Delta^s_{\mu} (x,2^n t)^2 \leq \frac{4 \delta^4}{\log 2}\, \theta^s_\mu({B_0})^2.
\end{equation*}
In particular,
\begin{equation*}
|\Delta^s_{\mu} (x,2^n t)| \leq \frac{2 \delta^2}{\sqrt{\log 2}}\, \theta^s_\mu({B_0})
\end{equation*}
for $n=-k-1, \dots, k+2$. This implies that
\begin{equation*}
\left| \frac{\mu(B(x,2^{k+3}t))}{(2^{k+3}t)^s}-\frac{\mu(B(x,t))}{t^s} \right| \leq 4(k+2) (2^{-k+1})^2 \,\theta^s_\mu({B_0})\leq \theta^s_\mu({B_0}).
\end{equation*}
Therefore,
\begin{equation*}
\frac{\mu(2 \delta^{-1}{B_0})}{(2 \delta^{-1}r_0)^s}\leq 2^{3s}\frac{\mu(B(x,2^{k+3}t))}{(2^{k+3}t)^s}\leq 2^{3s}\left( \frac{\mu(B(x,t))}{t^s}+2^s\, \theta^s_\mu(2B_0) \right) \leq 2^{5s+1}\theta^s_\mu(2B_0),
\end{equation*}
and so
\begin{equation}
\label{ddoubl}
\mu(2 \delta^{-1}{B_0})) \leq 2^{5s+1} \, \delta^{-s}\mu(2B_0).
\end{equation}
In the same way (in fact, just setting $\delta=1/2$) one easily deduces that
\begin{equation*}
\mu(4 {B_0}) \leq 2^{5s+2}\mu(2B_0).
\end{equation*}
Therefore we can apply Lemma \ref{lemcompactbol} to $2{B_0}$ and obtain
\begin{equation}
\label{2b0}
\int_{2\delta\,r_0}^{2\delta^{-1}\,r_0}\!\!
 \int_{x\in 2\delta^{-1} B_0} |\Delta^s_{\mu} (x,r)|\,d\mu(x)\,\frac{dr}r > \delta^{1/2}\,\frac{\mu(2B_0)^2}{(2r_0)^s}.
 \end{equation}
By Cauchy-Schwartz and \eqref{2b0}, it follows that
\begin{equation*}
\mu(2 \delta^{-1}B_0)\, \log(\delta^{-2})\,\int_{2\delta\,r_0}^{2\delta^{-1}\,r_0}\!\!
 \int_{x\in 2\delta^{-1} B_0} \Delta^s_{\mu} (x,r)^2\,d\mu(x)\,\frac{dr}r  >  \delta\,\frac{\mu(2 {B_0})^4}{(2r_0)^{2s}}.
 \end{equation*}
Finally using \eqref{ddoubl} we have,
$$\int_{2\delta\,r_0}^{2\delta^{-1}\,r_0}\!\!
 \int_{x\in 2\delta^{-1} B_0} \Delta^s_{\mu} (x,r)^2\,d\mu(x)\,\frac{dr}r \gtrsim \frac{\delta^{s+1}}{\log(\delta^{-2})}\, \theta^s_\mu({B_0})^2 \, \mu({B_0}).$$
\end{proof}


\section{Relationship with the $s$-Riesz transform for $0<s<1$}\label{secriesz}

It was shown in \cite{mpv} that for a {\em finite}  Radon measure $\mu$ in $\R^d$, we have
\begin{equation}\label{eqmpv}
\sup_{\ve>0} \int|R_\ve^s(\mu)(x)|^2d\mu(x)\approx \int \wolf(x)d\mu(x).
\end{equation}
In this section we extend this result to the case of non-finite Radon measures. In part, our motivation stems from the counterexample
that we will construct in Section \ref{counterexample} for the case $s=1$, 
which consists of a non-finite Radon measure for which the squared $L^2(\mu)$-norm of 
the $1$-Riesz transform of $\mu$ is not comparable to 
$\iint_0^\infty\Delta^1_\mu (x,r)^2\,\frac{dr}rd\mu(x)$.

The next proposition is stated in terms of the doubly truncated Riesz transform of $\mu$. 
Given $0<\ve_1<\ve_2$, this is defined as
$$\riesz(\mu)(x)= \int_{\ve_1<|y-x|\leq \ve_2} K^s(y-x)\,d\mu(y).$$

\begin{propo}\label{comparability}
Let $\mu$ be a Radon measure in $\R^d$ and $0<s<1$. Then the following statements hold:
\begin{enumerate}
\item[(a)] For every $\ve_1,\;\ve_2>0$,
\begin{equation}\label{rieszwolff}
\int|\riesz(\mu)(x)|^2d\mu(x)\le C\int \wolf(x)d\mu(x),
\end{equation}
with $C$ independent of $\ve_1$ and $\ve_2$.
\item[(b)] If $\mu$ is such that $\underset{r\to\infty}{\liminf}\frac{\;\mu(B(0,r))^3}{r^{2s}}<\infty$, then
\begin{equation}\label{wolffriesz}
\int\wolf(x)d\mu(x)\leq C\sup_{\ve_2>\ve_1>0}\int|\riesz(\mu)(x)|^2d\mu(x).
\end{equation}
\end{enumerate} 

\end{propo}

\begin{remark}
First, let us mention that it is easy to see that there exist non-finite measures $\mu$ with finite Wolff energy. 

Second, notice that in general \eqref{wolffriesz} does not hold without assuming
the finiteness of $\underset{r\to\infty}{\liminf}\frac{\;\mu(B(0,r))^3}{r^{2s}}$. Take for example 
 $\mu=\HH^1 |_{\R}$, then by antisymmetry one has $\int|\riesz(\mu)|^2d\mu=0$, while $\int\wolf d\mu=\infty$.

On the other hand, it is easy to check that 
\begin{equation}\label{eqww1}
\int\wolf(x)d\mu(x)<\infty \qquad\Rightarrow\; \quad \;\underset{r\to\infty}{\lim}\frac{\;\mu(B(0,r))^3}{r^{2s}}=0.
\end{equation}
 
 
\end{remark}

By combining Proposition \ref{comparability} and Theorem \ref{teomain1} we get the following corollary

\begin{coro}\label{corol}
Let $\mu$ be a Radon measure in $\R^d$ and let $0<s<1$. If $\mu$ is such that $\underset{r\to\infty}{\liminf}\frac{\;\mu(B(0,r))^3}{r^{2s}}<\infty$, 
then 
\begin{equation*}
\sup_{\ve_1,\ve_2>0}\int|\riesz(\mu)(x)|^2d\mu(x)\approx\int\wolf(x)d\mu(x)\approx \iint_0^\infty\Delta_\mu^s(x,r)^2\frac{dr}{r}d\mu(x).
\end{equation*}
\end{coro}

Before proving the proposition we need to recall the definition of balls with thin boundaries. Given $t>0$,
a ball $B(x,r)$ is said to have  $t$-thin boundary (or just thin boundary)
if
$$\mu\bigl(\{y\in B(x,2r):\dist(y,\partial B(x,r))\leq \lambda\,r\}\bigr) \leq t\,\lambda\,\mu(B(x,2r))$$
for all $\lambda>0$.
The following result is well known. For the proof (with cubes instead of balls) see Lemma 9.43 of \cite{Tolsa-llibre}, for
example.

\vv
\begin{lemma}\label{lemthin}
Let $\mu$ be a Radon measure on $\R^d$. Let $t$ be some constant big enough (depending only on $d$).
Let $B(x,r)\subset\R^d$ be any fixed ball. Then there exists $r'\in [r,2r]$ such that the ball $B(x,r')$ has $t$-thin boundary.
\end{lemma}

Now we turn to the proof of Proposition \ref{comparability}.

\begin{proof}[Proof of Proposition \ref{comparability}]
 If $\mu$ is a compactly supported Radon measure and $0<s<1$, then by \cite{mpv}, for any $r>0$,
 \begin{equation*}
 \int_{B(0,r)}|\riesze(\chi_{B(0,r)}\mu )(x)|^2d\mu(x)\lesssim\int\wolf(x)d\mu(x)
 \end{equation*}
 for all $\ve_1>0$ and $r>0$. Therefore,
 $$\int_{B(0,r_0)}|\riesz(\chi_{B(0,r)}\mu )(x)|^2d\mu(x)\lesssim\int\wolf(x)d\mu(x)$$
 for any $r_0<r$ and $\ve_1,\ve_2>0$.
 Since $$\lim_{r\to\infty}|\riesz(\chi_{B(0,r)}\mu )(x)|=|\riesz(\mu)(x)|$$ and for a fixed $r_0>0$ and $x\in B(0,r_0),$ 
 $$|\riesz(\mu)(x)|\le\frac{\mu(B(0,r_0+\ve_2))}{\ve_1^s}=C_{r_0,\ve_1,\ve_2},$$ the dominated convergence theorem proves \eqref{rieszwolff}.
 
 Now we deal with inequality \eqref{wolffriesz}.  
 Clearly we may assume that 
 $$\sup_{\ve_1,\ve_2}\int|\riesz(\mu)(x)|^2d\mu(x)<\infty,$$
since otherwise the statement (b) is trivial.
 Using \cite{mpv}, given $r>0$ and taking $\ve_2 =2r$ we get 
 \begin{equation}\label{eqav1}
 \begin{split}
 \int_{B(0,r)}\wolfr(x)\,d\mu(x)&\lesssim\sup_{\ve_1>0}\int_{B(0,r)}|\riesz(\chi_{B(0,r)}\mu )(x)|^2d\mu(x)\\
 &\lesssim \sup_{\ve_1>0}\int|\riesz(\mu)(x)|^2d\mu(x)\\
 &\quad +\sup_{\ve_1>0}\int_{B(0,r)}|\riesz(\chi_{B(0,r)^c}\mu)(x)|^2d\mu(x).
 \end{split}
 \end{equation}
 
 We claim that if $B(0,r)$ has thin boundary, then
\begin{equation}\label{eqclaimm}
 |\riesz(\chi_{B(0,r)^c}\mu)(x)|\lesssim \theta_\mu^s(B(0,3r))
 \qquad \mbox{for $\ve_2=2r$ and $x\in B(0,r)$.}
\end{equation}
Assuming this for the moment, we get
$$
\int_{B(0,r)}|\riesz(\chi_{B(0,r)^c}\mu)(x)|^2\,d\mu(x)
\lesssim \theta_\mu^s(B(0,3r))^2\,\mu(B(0,r)),
$$
and thus
\begin{equation}\label{eqaa12}
\begin{split}
 \int_{B(0,r)}&\wolfr(x)\,d\mu(x)\lesssim \\ &\sup_{\ve_2>\ve_1>0}\int|\riesz(\mu)(x)|^2d\mu(x) + \theta_\mu^s(B(0,3r))^2\,\mu(B(0,3r)).
\end{split}
\end{equation}
By the assumption in (b), there exists a sequence $r_k\to\infty$ such that
$$\sup_{k>0}\theta_\mu^s(B(0,r_k))^2\,\mu(B(0,r_k))<\infty.$$
By Lemma \ref{lemthin}, for each $k$ there exists some $\wt r_k\in [\frac16r_k,\frac13r_k]$ such that the ball $B(0,\wt r_k)$
has thin boundary. Since
$\theta_\mu^s(B(0,3\wt r_k))^2\,\mu(B(0,3\wt r_k)) \lesssim \theta_\mu^s(B(0,r_k))^2\,\mu(B(0,r_k))$, from \rf{eqaa12}
we deduce
$$ \int_{B(0,\wt r_k)}\wolfrk(x)\,d\mu(x)\lesssim \sup_{\ve_2>\ve_1>0}\int|\riesz(\mu)(x)|^2d\mu(x) + 
\theta_\mu^s(B(0,r_k))^2\,\mu(B(0,r_k)).$$
Letting $k\to\infty$, we obtain
\begin{equation}\label{eqdk99}
\int\wolf(x)d\mu(x)\lesssim \sup_{\ve_2>\ve_1>0}\int|\riesz(\mu)(x)|^2d\mu(x)+\underset{r\to\infty}{\liminf}\frac{\;\mu(B(0,r))^3}{r^{2s}}.
\end{equation}
By the assumptions in (b) the right hand above is finite and thus
 $\int\wolf(x)d\mu(x)<\infty$, which in turn implies that 
$\underset{r\to\infty}{\lim}\frac{\;\mu(B(0,r))^3}{r^{2s}}=0$ by \rf{eqww1}.
Hence the statement (b) of the proposition follows from \rf{eqdk99}.

It remains to prove \rf{eqclaimm}. For $x\in B(0,r)$ and $\ve_2=2r$, we have
\begin{align*}
|\riesz(\chi_{B(0,r)^c}\mu)(x)| &\leq \int_{B(0,3r)\setminus B(0,r)} \frac1{|y-x|^s}\,d\mu(y)\\
&\lesssim \sum_{k\geq 1 } \frac{\mu(B(x,2^{-k}r)\cap B(0,3r)\setminus B(0,r))}{(2^{-k}r)^{s}}.
\end{align*}
Note now that if $B(x,2^{-k}r)\cap B(0,3r)\setminus B(0,r)\neq\varnothing$, then 
$$B(x,2^{-k}r)\cap B(0,3r)\subset  B(0,3r)\cap U_{2^{-k+1}}(\partial B(0,r)),$$ where $U_\delta(A)$ stands for the $\delta$-neighborhood of $A$. Thus, in any case we have
$$\mu(B(x,2^{-k}r)\cap B(0,3r)\setminus B(0,r)) \lesssim 2^{-k}\,\mu(B(0,3r)),$$
because $B(0,r)$ has thin boundary. Therefore,
$$|\riesz(\chi_{B(0,r)^c\mu})(x)|\lesssim \sum_{k\geq 1 }  2^{-k(1-s)} \frac{\mu(B(0,3r))}{r^s} \lesssim \theta_\mu^s(B(0,3r)),$$
as claimed.
\end{proof}

\section{Counterexample in the integer case $s=1$}\label{counterexample}

In this section we give an example of an infinite measure in the plane such that the quantities
\begin{equation}\label{eq333}
\int|R^1(\mu)(x)|^2d\mu(x),\quad\int W^\mu_{\frac23,\frac32}(x)d\mu(x),\quad  \iint_0^\infty\Delta_\mu^1(x,r)^2\frac{dr}{r}d\mu(x)
\end{equation}
are not comparable, in contrast to the result stated in Corollary \ref{corol} for $0<s<1$.
Our example consists of a measure $\mu$ with linear growth (i.e.\ growth $1$), infinite Wolff energy, for which the 
squared  $L^2(\mu)$-norm of the 
$1$-Riesz transform with respect to $\mu$ is finite and much bigger than $\iint_0^\infty\Delta_\mu^1(x,r)^2\frac{dr}{r}d\mu(x)$. 
We think that this fact is quite surprising, because for  
general measures $\mu$ with linear growth (i.e.,\ with growth $1$) in the complex plane, it has been recently shown in \cite{tolsaarxiv} that
$$\int_Q|R^1_\mu\chi_Q|^2d\mu\le C\mu(Q)\qquad \mbox{for every square $Q\subset \C$}$$
if and only if
$$\int_Q\int_0^\infty\Delta_{\chi_Q\mu}^1(x,r)^2\frac{dr}{r}d\mu(x)\le C'\mu(Q)\qquad \mbox{for every square $Q\subset \C$}.$$

Now we turn to the construction of the measure for our counterexample. Consider the curve $\Gamma_\alpha$ as in Figure \ref{fig1}, for $0<\alpha \leq \pi/4$. 
In particular, $\Gamma_\alpha \subset \R^2$ can be realized as the graph of the piecewise linear function $f: \R \ra \R$ defined by 
$$f(x) = \begin{cases} 0 &\mbox{if } x \in (-\infty,-1/2] \cup [1/2, \infty),  \\
\tan \a \,(x+1/2) & \mbox{if } x \in (-1/2,0],
\\
-\tan \a \,(x-1/2) & \mbox{if } x \in (0,1/2). \end{cases}$$
We set \begin{itemize}
\item $L_1^\a=(-\infty, 1/2] \times \{0\}$, 
\item $T_1^\a=\{(x,f(x)) \in \R^2:x \in[-1/2,0]\}$,
\item $T_2^\a=\{(x,f(x)) \in \R^2:x \in[0,1/2]\}$,
\item $L_2^\a=[1/2,\infty)\times \{0\}$
\item $T^\a=T_1^\a \cup T_2^\a$.
\end{itemize}

We will show that, for the $1$-dimensional Hausdorff measure $\mu=\ha^1|_{\Gamma_\a}$, the three quantities in \rf{eq333} are not comparable.
 Note that, strictly speaking, this fact cannot be consider as a counterexample to
Corollary \ref{corol} for the case $s=1$, since the assumption $\underset{r\to\infty}{\liminf}\frac{\;\mu(B(0,r))^3}{r^{2}}<\infty$ does not hold.
\vv

\begin{propo}
 \label{counter0}
 Let $\a \in (0, \pi/4]$. Then 
\begin{equation}\label{counterexample1}
 \iint_0^\infty \Delta^1_{\ha^1|_{\Gamma_\a}}(x,r)^2 \frac{dr}r\, d\ha^1|_{\Gamma_\a}(x) \lesssim \sin^4\a.
\end{equation}
Also,
\begin{equation}\label{counterexample2}
\sup_{\ve_1,\ve_2>0}\int|\rieszz(\mu)(x)|^2d\mu(x)\approx\sin^2\a.
\end{equation}
\end{propo}

It is clear that, letting $\alpha\to0$, 
we will get
$$ \iint_0^\infty \Delta^1_{\ha^1|_{\Gamma_\a}}(x,r)^2 \frac{dr}r\ll 
\sup_{\ve_1,\ve_2>0}\int|\rieszz(\mu)(x)|^2d\mu(x).$$
On the other hand, it is easy to check that $\int W^{\ha^1|_{\Gamma_\a}}_{\frac23,\frac32}(x)d\ha^1|_{\Gamma_\a}(x)=\infty$.

To prove Proposition \ref{counter0}, we consider the auxiliary $1$-AD regular measure on $\Gamma_\a$:
$$\mu_\a= \cos\a\, \ha^1 |_{T^\a}+\ha^1 |_{ \Gamma_\a \setminus T_\a},$$
for which the following holds:

\begin{lemma}
 \label{counter1}
 Let $\a \in (0, \pi/4]$. Then
\begin{equation*}
\iint_0^\infty \Delta^1_{\mu_\a}(x,r)^2 \frac{dr}r\, d\mu_\a(x) \lesssim \sin^4\a.
\end{equation*}
\end{lemma}

\begin{figure}
\centering
\includegraphics[scale = 0.6]{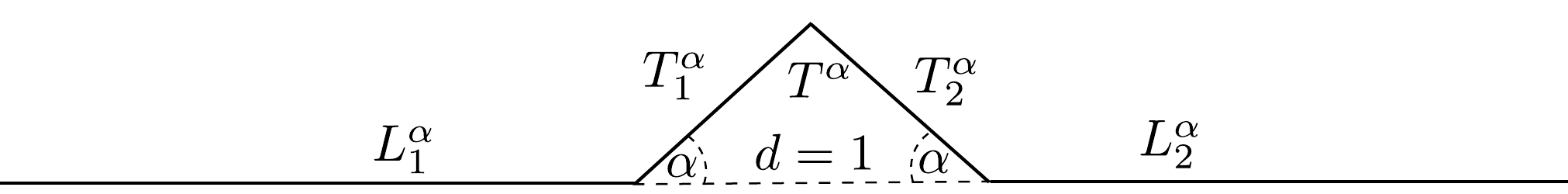}
\caption{The curve $\Gamma_\alpha$.}
\label{fig1}
\end{figure}

Using Lemma \ref{counter1}, we are now able to prove Proposition \ref{counter0}.

\begin{proof}[Proof of Proposition \ref{counter0}]
To show \eqref{counterexample1}, notice that 
\begin{equation}\label{difference}
\begin{split}
\Delta^1_{\ha^1|_{\Gamma_\a}}(x,r)-\Delta^1_{\mu_\a}(x,r)&\le c\,(1-\cos\a)\,\frac{\ha^1(T^\a\cap B(x,2r))}{r}\\&\le c\;\a^2\frac{\ha^1(T^\a\cap B(x,2r))}{r}.
\end{split}
\end{equation}
Hence,
\begin{equation*}
\begin{split}
\biggl(\int_{1/10}^\infty\int\Delta^1_{\ha^1|_{\Gamma_\a}}(x,r)^2& d\ha^1|_{\Gamma_\a}(x)\frac{dr}{r}\biggr)^{1/2}\\
& \lesssim \left(\int_{1/10}^\infty\int\left(\Delta^1_{\ha^1|_{\Gamma_\a}}(x,r)-\Delta^1_{\mu_\a}(x,r)\right)^2d\ha^1|_{\Gamma_\a}(x)\frac{dr}{r}\right)^{1/2}\\&\quad+
\left(\int_{1/10}^\infty\int\Delta^1_{\mu_\a}(x,r)^2d\ha^1|_{\Gamma_\a}(x)\frac{dr}{r}\right)^{1/2}=A+B,
\end{split}
\end{equation*}
the last identity being a definition for $A$ and $B$. By Lemma \ref{counter1}, we have $B\le c\,\a^2$. 
Concerning $A$, using \eqref{difference} and the linear growth of $\ha^1|_{\Gamma_\a}$, we get
\begin{equation*}
\begin{split}
A^2&\le\a^4\int_{r\ge 1/10}\int_{|x|\le 10r}\frac{\ha^1(T^\a\cap B(x,2r))^2}{r^2}d\ha^1|_{\Gamma_\a}(x)\frac{dr}{r}\\\\&\le\a^4\int_{r\ge 1/10}\int_{|x|\le 10r}\frac{d\ha^1|_{\Gamma_\a}(x)}{r^3}dr\le
c\a^4\int_{r\ge 10}\frac{dr}{r^2}\le c\a^4.
\end{split}
\end{equation*}

Now we write
\begin{equation*}
\begin{split}
 \biggl(\int_0^{1/10}\int\Delta^1_{\ha^1|_{\Gamma_\a}}(x,r)^2\,&d\ha^1|_{\Gamma_\a}(x)\frac{dr}{r}\biggr)^{1/2}\\
 & \lesssim \left(\int_0^{1/10}\int\left(\Delta^1_{\ha^1|_{\Gamma_\a}}(x,r)-\Delta^1_{\mu_\a}(x,r)\right)^2d\ha^1|_{\Gamma_\a}(x)\frac{dr}{r}\right)^{1/2}\\&\quad+
\left(\int_0^{1/10}\int\Delta^1_{\mu_\a}(x,r)^2d\ha^1|_{\Gamma_\a}(x)\frac{dr}{r}\right)^{1/2}=C+D,
\end{split}
\end{equation*}
the last identity being the definition of $C$ and $D$. Again by Lemma \ref{counter1}, $D\le c\,\a^2$. 
To estimate $C$, 
we consider the vertices 
$\{z_a\}=L_1^{\a}\cap T_1^{\a}$, $\{z_b\}=T_1^\a\cap T_2^\a$ and $\{z_c\}=T_2^\a\cap L_2^\a$. 
It is easy to check that $\Delta^1_{\ha^1|_{\Gamma_\a}}(x,r)-\Delta^1_{\mu_\a}(x,r)$ vanishes unless 
$z_a\in B(x,2r)$, $z_b\in B(x,2r)$ or 
$z_c\in B(x,2r)$. Then we split the integral in $C$ into three integrals according to the three preceding  cases.
  Since they are treated similarly, we will deal only with the first one. Using again \eqref{difference}, we get
\begin{multline*}
\int_0^{1/10}\int_{z_a\in B(x,2r)}\Delta^1_{\ha^1|_{\Gamma_\a}}(x,r)^2\,d\ha^1|_{\Gamma_\a}(x)\frac{dr}{r}  \\
\lesssim
\a^4\int_0^{1/10}\int_{|x-z_a|< 2r}\left(\frac{\ha^1(T^\a\cap B(x,2r))}{r}\right)^2 d\ha^1|_{\Gamma_\a}(x)\,\frac{dr}{r}\le c\a^4.
\end{multline*}
Gathering all the preceding estimates, \rf{counterexample1} follows.

The proof of \eqref{counterexample2} can be obtained either by \cite[Corollary 1.4]{tolsafunct} or by more elementary methods
(which we leave for the reader).
\end{proof}



\vv

\begin{proof}[Proof of Lemma  \ref{counter1}]
We fix $\a \in (0, \pi/4]$. To simplify notation write $\mu$, $\Gamma$, $T$, $L$ and so on, instead of $\mu_\a$, $\Gamma_\a$, $T^\a$, $L_ \a$. 

To estimate $\Delta_\mu^1(x,r)$ for $x\in\Gamma$ and $r>0$, note first that $\Delta_\mu(x,r)$ vanishes if one of the following conditions holds:
\begin{itemize}
\item $B(x,2r)\cap \Gamma\subset L_i$ for $i=1$ or $2$,
\item $B(x,2r)\cap \Gamma\subset T_i$ for $i=1$ or $2$.
\item $x\in L_1\cup L_2$ and $T\subset B(x,r)$.
\end{itemize}
In this case, we set $(x,r)\in Z$.

In the case $(x,r)\not\in Z$, we write
\begin{equation}
\label{Dd}
\begin{split}
|\Delta^1_{\mu}(x,r)|&\leq \left| \frac{\mu(B(x,r))-2r}{r}\right|+\left| \frac{\mu(B(x,2r))-4r}{2r}\right|\\
&=:\delta_\mu(B(x,r))+\delta_\mu(B(x,2r)).
\end{split}
\end{equation}
We claim that 
\begin{equation}\label{eqclaim000}
\delta_\mu(B(x,r))\lesssim 
\min\biggl(1,\frac1{r^2}\biggr)\,\sin^2\alpha
\qquad \mbox{for $(x,r)\not \in Z$}.
\end{equation}
In fact, this holds for all $x\in\Gamma$ and $r>0$, but we only need to prove it for $(x,r)\not\in Z$.
Let us see that the lemma follows from this estimate. 
We write
\begin{align*}
\iint_0^\infty \Delta^1_{\mu}(x,r)^2\, \frac{dr}r\, d\mu(x) & \lesssim 
 \iint_{(x,r)\in (\Gamma\times (0,\infty))\setminus Z}  \delta_\mu(B(x,r))^2\,
\frac{dr}r\, d\mu(x)\\
& \lesssim
\sin^4\a \iint_{(x,r)\in (\Gamma\times (0,\infty))\setminus Z}  \min\biggl(1,\frac1{r^4}\biggr)
\frac{dr}r\, d\mu(x).
\end{align*}
So, to prove the lemma it is enough to show that the last integral does not exceed some absolute constant.
To this end, denote $\{z_1\}= L_1\cap T_1$, $\{z_2\}= L_2\cap T_2$, and $\{z_0\}= T_1\cap T_2$, and set
$$B_1=\{(x,r) \in\Gamma\times (0,\infty):
z_i\in B(x,2r) \text{ for some }i=0,1,2\text{ and }T\not\subset B(x,r)\} $$
and 
$$B_2=\{(x,r) \in\Gamma\times (0,\infty): x \in T \text{ and }
T\subset B(x,r)\}.$$
Notice that $\Gamma\times (0,\infty))\setminus Z= B_1 \cup B_2$ and if $(x,r) \in B_2$ then $r>1/2$. We then have
\begin{equation*}
\begin{split}
&\iint_{(x,r)\in (\Gamma\times (0,\infty))\setminus Z}  \min\biggl(1,\frac1{r^4}\biggr)
\frac{dr}r\, d\mu(x)\\
&\quad=\iint_{B_1}  \min\biggl(1,\frac1{r^4}\biggr)
\frac{dr}r\, d\mu(x)+\iint_{B_2}  \min\biggl(1,\frac1{r^4}\biggr)
\frac{dr}r\, d\mu(x)\\ &\quad= \sum_{i=0}^2\iint_{\begin{subarray}\,(x,r)\in \Gamma\times (0,\infty)\\
z_i\in B(x,2r),\,T\not\subset B(x,r)\end{subarray}} \min\biggl(1,\frac1{r^4}\biggr)\,\frac{dr}r\, d\mu(x)+\iint_{B_2}  \min\biggl(1,\frac1{r^4}\biggr)
\frac{dr}r\, d\mu(x).
\end{split}
\end{equation*}
For $0\leq i\leq2$ we consider two cases according to whether $r<1/2$ or not. We set
\begin{align*}
\iint_{\begin{subarray}\,(x,r)\in \Gamma\times (0,1/2)\\
z_i\in B(x,2r),\,T\not\subset B(x,r)\end{subarray}} &\min\biggl(1,\frac1{r^4}\biggr)\,\frac{dr}r\, d\mu(x)
= \iint_{\begin{subarray}\,(x,r)\in \Gamma\times (0,1/2)\\
z_i\in B(x,2r),\,T\not\subset B(x,r)\end{subarray}} \frac{dr}r\, d\mu(x).
\end{align*}
Integrating first with respect to $x$, taking into account that $|x-z_i|<2r$, the last integral is bounded by
$$c\int_{r\in (0,1/2)} r\,\frac{dr}r \approx1.$$
For $r>1/2$ we have
\begin{equation}\label{eqahj31}
\iint_{\begin{subarray}\,(x,r)\in \Gamma\times [1/2,\infty)\\
z_i\in B(x,2r),\,T\not\subset B(x,r)\end{subarray}} \min\biggl(1,\frac1{r^4}\biggr)\,\frac{dr}r\, d\mu(x)
= \iint_{\begin{subarray}\,(x,r)\in \Gamma\times [1/2,\infty)\\
z_i\in B(x,2r),\,T\not\subset B(x,r)\end{subarray}} \frac{dr}{r^5}\, d\mu(x).
\end{equation}
It is easy to check that for $(x,r)$ in the domain of integration above we have
$$r-c\leq|z_i-x|\leq r + c$$
for some absolute constant $c$, which implies that
$$\mu\bigl(\{x:(x,r)\in \Gamma\times [1,\infty),
z_i\in B(x,2r),\,T\not\subset B(x,r)\}\bigr)\leq c'.$$
Thus the integral in \rf{eqahj31} is bounded by
$$c\int_{r>1/2} \frac{dr}{r^5}\lesssim1.$$
Therefore we have shown that
\begin{equation}
\label{b1}
\iint_{B_1}  \min\biggl(1,\frac1{r^4}\biggr)
\frac{dr}r\, d\mu(x) \lesssim 1.
\end{equation}

In the same way, for $r>1/2$
$$\mu\bigl(\{x:(x,r)\in B_2\}\bigr)\leq c'',$$
hence
\begin{equation}
\label{b2}
\iint_{B_2}  \min\biggl(1,\frac1{r^4}\biggr)
\frac{dr}r\, d\mu(x) \lesssim 1.
\end{equation}
Thus \eqref{b1} and \eqref{b2} imply
$$\iint_{(x,r)\in (\Gamma\times (0,\infty))\setminus Z}  \min\biggl(1,\frac1{r^4}\biggr)
\frac{dr}r\, d\mu(x)\lesssim1,$$
as wished.

\vv
To prove the claim \rf{eqclaim000} we distinguish several cases:

\vv

\textbf{Case A:} $(x,r)\not \in Z$ and $x \in T$.\\

\textit{Subcase A1:} $B(x,2r)\cap\Gamma \subset T$. \\
\vv 
We note that this subcase is possible only for $r<2$. We write
\begin{equation}
\label{cosdens}
\begin{split}
\delta_\mu(B(x,r))&\leq \left| \frac{\mu(B(x,r))}{r}-2 \, \cos \a\right|+2(1-\cos \a)\\
&\approx \left| \frac{\mu(B(x,r))}{r}-2 \, \cos \a\right|+\sin^2 \a.
\end{split}
\end{equation}

Since $(x,r) \notin Z$ we have that $B(x,2r) \cap T_i \neq \varnothing$ for both $i=1,2$. Without loss of generality let $x\in T_2$ and set $\{v_0\}=T_1 \cap T_2,$ $\{v_i\}=\partial B(x,2r) \cap T_i$, for $i=1,2$, $\{p\} = \partial B(x,2r) \cap L_{v_0v_2} \stm \{v_2\}$, where $L_{v_0v_2}$ denotes the line crossing $v_0$ and $v_2$. Let also $p'$ be the point of intersection of $L_{v_0v_1}$ and the perpendicular line to $L_{v_0v_2}$ which passes through $p$. See also Figure \ref{a1}. 
\begin{figure}
\centering
\begin{minipage}{.45\textwidth}
  \centering
  \includegraphics[width=1\linewidth]{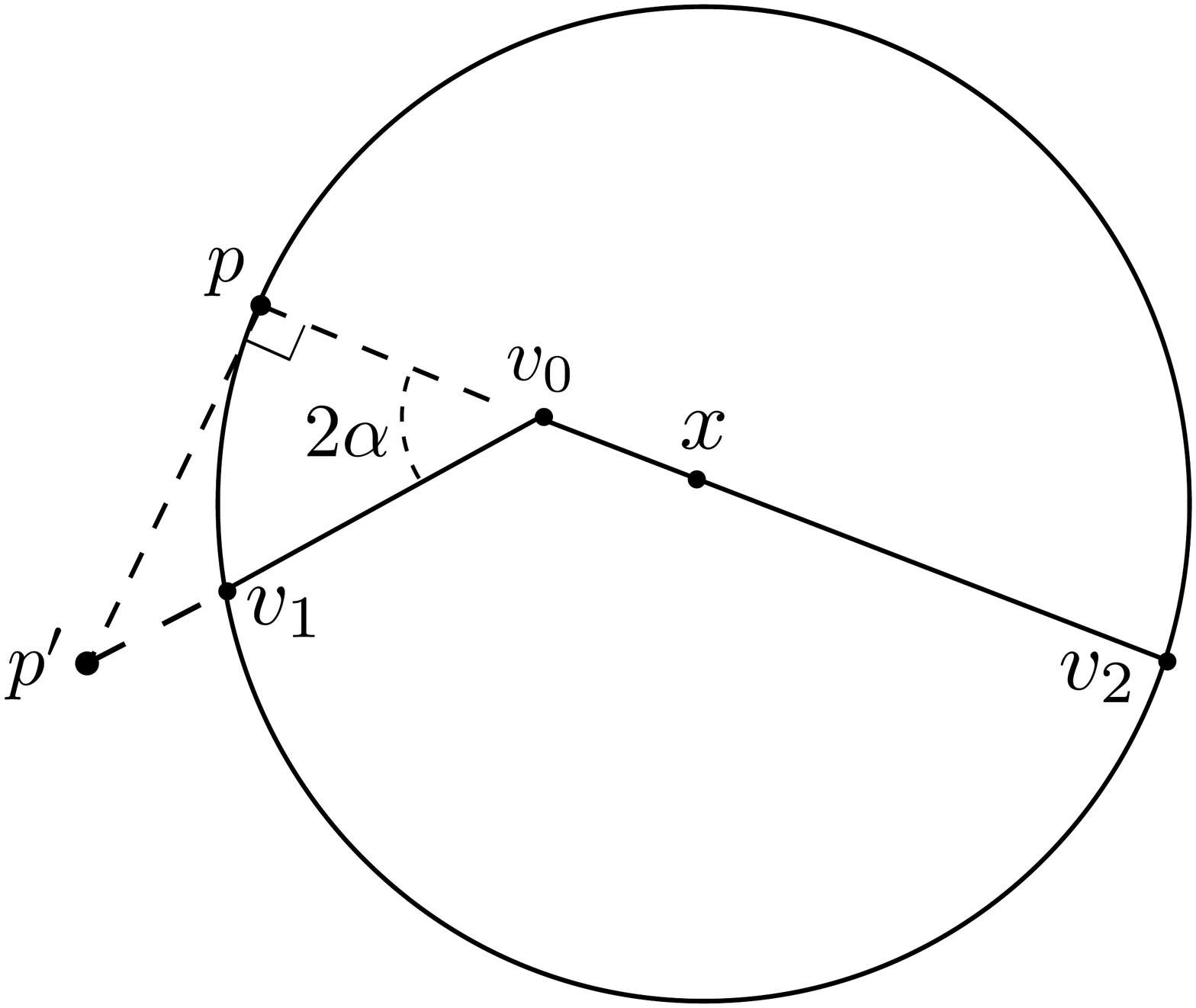}
  \caption{The case $A1$}
\label{a1}
 
\end{minipage}%
\begin{minipage}{.45\textwidth}
  \centering
  \includegraphics[width=1\linewidth]{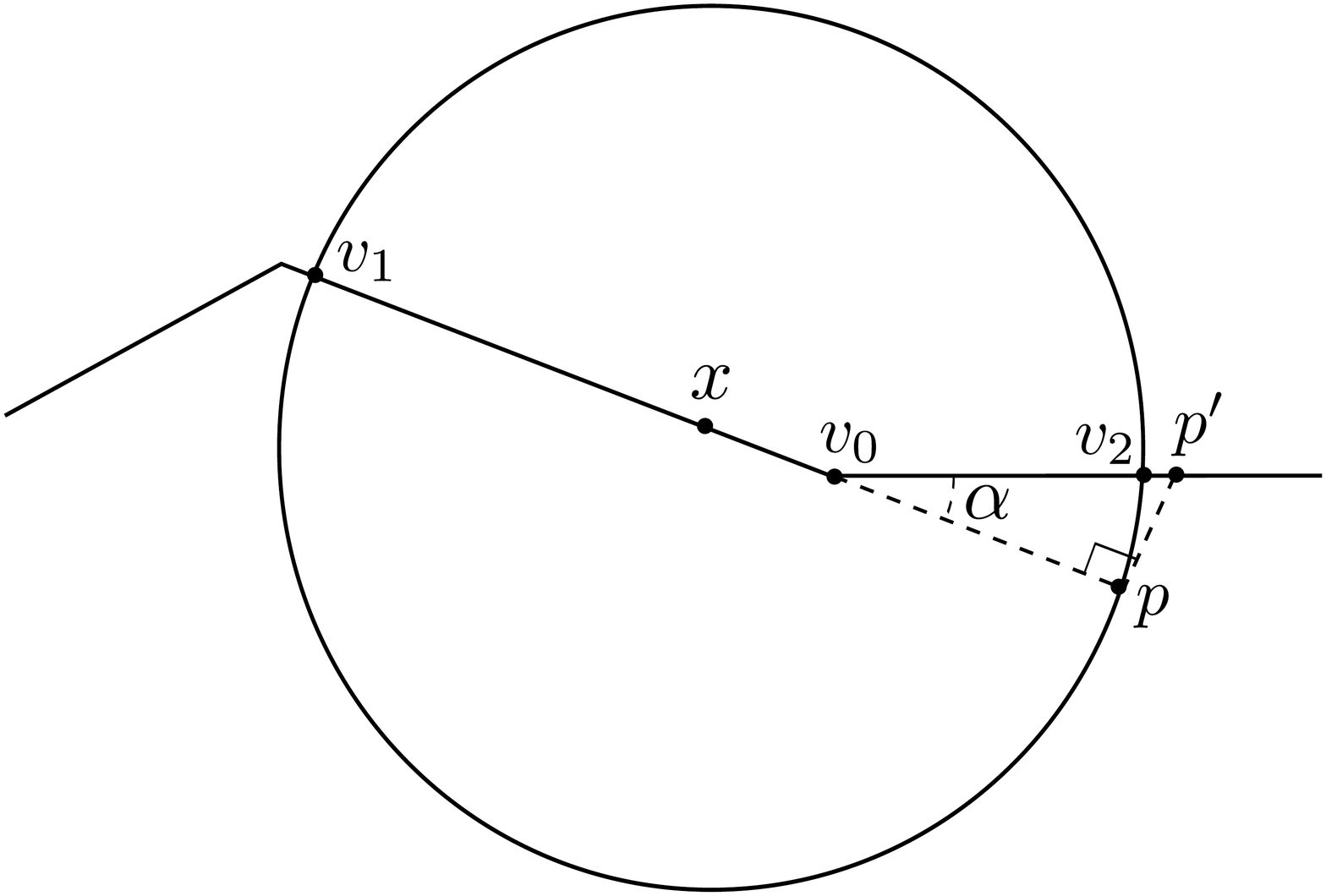}
  \caption{The case $A2$}
\label{a2}
\end{minipage}
\end{figure}

We have
\begin{equation*}
\begin{split}
|\mu(B(x,r))-2\, r\,\cos \a|&=|\cos \a \, (d(v_1,v_0)+d(v_0,v_2))-\cos \a (d(v_0,v_2)+d(v_0,p))|\\
&\approx |d(v_1,v_0)-d(v_0,p))|.
\end{split}
\end{equation*}
Observing that 
\begin{equation*}
\begin{split}
|d(v_1,v_0)-d(v_0,p)|&\leq d(v_0,p')-d(v_0,p)=d(v_0,p')-d(v_0,p') \cos(2\a)\\
&\approx \sin^2(2\a) \, d(v_0,p') \lesssim r\sin^2 \a 
\end{split}
\end{equation*}
 and recalling \eqref{cosdens}, we deduce that $\delta_\mu(B(x,2r)) \lesssim \sin^2 \a$.



\vv

\textit{Subcase A2:} $B(x,2r)\cap\Gamma \subset T_i \cup L_i$ for $i=1$ or $2$.\\
\vv
Without loss of generality we can assume that $i=2$. We consider the following points:
 $\{v_0\}=L_2 \cap T_2$, $\{v_1\}=\partial B(x,2r) \cap T_2$, $\{v_2\}=\partial B(x,2r) \cap L_2$, and $\{p\} = \partial B(x,2r) \cap L_{v_0v_1} \stm \{v_1\}$. Let also $p'$ be the point of intersection of $L_{v_0v_2}$ and the perpendicular line to $L_{v_0v_1}$ which passes through $p$. See also Figure \ref{a2}. 

We have
\begin{equation*}
\begin{split}
|\mu(B(x,r))-4\, \cos \a \,r|&=|\cos \a \, d(v_1,v_0)+d(v_0,v_2)-\cos \a (d(v_0,v_1)+d(v_0,p))|\\
&\leq (1-\cos \a) d(v_0,v_2) +\cos \a \,(d(v_0,v_2)-d(v_0,p))\\
& \lesssim r\sin^2 \a +\cos \a \,(d(v_0,v_2)-d(v_0,p)).
\end{split}
\end{equation*}
As in the previous subcase,
\begin{equation*}
\begin{split}
d(v_0,v_2)-d(v_0,p)&\leq d(v_0,p')-d(v_0,p)=d(v_0,p')-d(v_0,p') \cos \a\\
&\approx d(v_0,p')\sin^2 \a   \lesssim r\sin^2 \a,
\end{split}
\end{equation*}
hence we deduce that $\delta_\mu(B(x,2r)) \lesssim \sin^2 \a$.  

\vv
\textit{Subcase A3:} $B(x,r)\cap L_1 \neq \emptyset$ and $B(x,r)\cap L_2 \neq \emptyset$.\\
\vv
If $r\leq 1$, combining the arguments from the two previous cases, it follows that $\delta_\mu(B(x,r)) \lesssim \sin^2 \a$.

We now consider the case  when $r >1$ and without loss of generality we assume that $x=(w,h) \in T_1$. Given two lines $L,L'$, we denote by $\meas(L,L')$ the smallest angle between $L$ and $L'$. Let $\{v_1\}= \partial B(x,r) \cap L_1$ and let $\theta_1= \meas (L_{x,v_1},L_1)$. Then it follows easily that
$$\delta_\mu (B(x,r)) \lesssim \sin^2\theta_1 \approx \left( \frac{h}{r} \right)^2 \leq \frac{\sin^2 \alpha}{r^2},$$
see also Figure \ref{fig3}.

\vv
\textbf{Case B:} $(x,r) \notin Z \text{ and }x \notin T$.\\
\begin{figure}
\centering
\includegraphics[scale = 0.6]{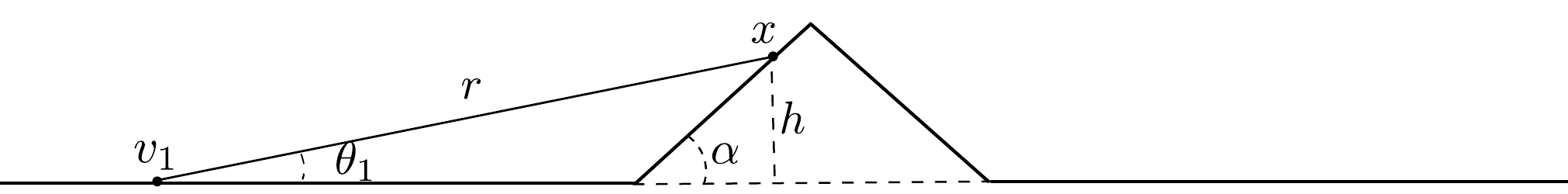}
\caption{The curve $\Gamma_\alpha$.}
\label{fig3}
\end{figure}

Let $\{v_1\}=\partial B(x,r)  \cap T$, then if $\theta_1=\meas(L_{xV_1}, L_2)$  we get 
\begin{equation}
\label{b1de}
\delta_\mu (B(x,r)) = 1-\cos \theta_1 \approx \sin^2 \theta_1.
\end{equation}
Since $\theta_1 < \alpha$, we have that $\sin^2 \theta_1 < \sin^2 \alpha$. Moreover if $r>1$, as in subcase A3, we get that $\sin^2 \theta_1 < \frac{\sin^2 \alpha}{r^2}$. Therefore,
$$\delta_\mu (B(x,r)) \lesssim \min\biggl(1,\frac1{r^2}\biggr)\,\sin^2\alpha.$$

Thus \eqref{eqclaim000} follows and the proof of the lemma is complete. 

\end{proof}

\vvv
We wish now to compare the integral $\iint_0^\infty\Delta_\mu^1(x,r)^2\frac{dr}{r}d\mu(x)$ to
the analogous one involving the so called $\beta$-numbers of Peter Jones, which play a key role
in the theory of the so called quantitative rectifiability (see \cite{Jones}, \cite{DS1} and \cite{DS2},
for example).
Given a Radon measure $\mu$ in $\R^d$, the Jones' $\beta$-numbers are defined as follows. For $x \in \supp(\mu)$ and $r>0,$ set, for $1\leq p<\infty$,
$$\beta^\mu_p(B(x,r))= \inf_L \left(\int_{B(x,r)} \frac{\dist(y,L)^p}{r^{p+1}}d \mu (y)\right)^{1/p},$$
and
$$\beta^\mu_\infty(B(x,r))= \inf_L \sup_{y \in B(x,r) \cap \supp(\mu)} \frac{\dist(y,L)}{r},$$
where in both cases the infimum is taken over all lines $L\subset\R^d$. 

By \cite[Theorem 6]{dorronsoro} (in the case $1\leq p<\infty$), for the measure $\mu_\a$ of  Proposition \ref{counter0} we have
$$\iint_0^\infty \beta^{\mu_\a}_p(B(x,r))^2\frac{dr}{r}d\mu_\a(x)\approx \|f'\|_2^2\approx\sin^2\a.$$ 
This also holds for $p=\infty$, by \cite{Jones}. 
So together with Proposition \ref{counter0}, this yields
$$\iint_0^\infty \Delta_{\mu_\a}^1(x,r)^2\frac{dr}{r}d\mu_\a(x)\ll
\iint_0^\infty \beta^{\mu_\a}_p(B(x,r))^2\frac{dr}{r}d\mu_\a(x)\qquad \mbox{as $\alpha\to0$,}$$
for all $1\leq p \leq\infty$.



\vvv

\end{document}